\documentclass[a4paper,11pt]{article}

\usepackage{amsmath, amssymb, amsthm}
\usepackage{graphics}

\usepackage{psfrag, booktabs}
\usepackage{color, epsfig, float}
\usepackage[active]{srcltx}
\usepackage[bookmarksopen, colorlinks, linkcolor = blue, urlcolor = red,
            citecolor = red, menucolor = blue]{hyperref}

\newtheorem{theorem}{Theorem}[section]
\newtheorem{lemma}[theorem]{Lemma}

\newtheorem{propo}[theorem]{Proposition}

\newtheorem{remark}[theorem]{Remark}

\newcommand\F{\mathbf{F}}

\newcommand\summ{\textstyle\sum\limits}

\newcommand\R{\mathbb{R}}

\newcommand\N{\mathbb{N}}

\DeclareMathOperator{\nr}{\mathcal N}
\DeclareMathOperator{\ra}{\mathcal R}

\newcommand\norm[1]{\|#1\|}
\newcommand\bracket[1]{\langle#1\rangle}
\newcommand\set[1]{\{#1\}}

\begin{document}

\title{On Levenberg-Marquardt-Kaczmarz iterative methods for solving
systems of nonlinear ill-posed equations} \setcounter{footnote}{1}

\author{
J.~Baumeister%
\thanks{Fachbereich Mathematik, Johann Wolfgang Goethe Universit\"at,
       Robert--Mayer--Str. 6--10, 60054 Frankfurt am Main, Germany
        \href{mailto:baumeist@math.uni-frankfurt.de}{\tt baumeist@math.uni-frankfurt.de}.}
\and
B.~Kaltenbacher%
\thanks{Universit\"at Stuttgart, Fachbereich Mathematik, Institut f\"ur
       Stochastik und Anwendungen, Pfaffenwaldring 57, 70569 Stuttgart, Germany
        \href{mailto:Barbara.Kaltenbacher@mathematik.uni-stuttgart.de}
        {\tt Barbara.Kaltenbacher@mathematik.uni-stuttgart.de}.}
\and
A.~Leit\~ao%
\thanks{Department of Mathematics, Federal University of St. Catarina,
        P.O. Box 476, 88040-900 Florian\'opolis, Brazil
        \href{mailto:aleitao@mtm.ufsc.br}{\tt aleitao@mtm.ufsc.br}.} }
\date{\small \today}

\maketitle

\begin{abstract}
In this article a modified Levenberg-Marquardt method coupled with a Kaczmarz
strategy for obtaining stable solutions of nonlinear systems of ill-posed
operator equations is investigated. We show that the proposed method is a
convergent regularization method. Numerical tests are presented for a non-linear
inverse doping problem based on a bipolar model.
\end{abstract}

\noindent {\small {\bf Keywords.} Nonlinear systems; Ill-posed
equations; Regularization; Levenberg-Marquardt method; Kaczmarz method.}
\medskip

\noindent {\small {\bf AMS Classification:} 65J20, 47J06.}

% --------------------------------------------------------------------
\section{Introduction} \label{sec:intro}

In this paper we propose a new method for obtaining regularized
approximations of systems of nonlinear ill-posed operator equations.

The \textit{inverse problem} we are interested in consists of
determining an unknown quantity $x \in X$ from the set of
data $(y_0, \dots,  y_{N-1}) \in Y^N$, where $X$, $Y$ are Hilbert
spaces and $N \geq 1$ (the case $y_i \in Y_i$ with possibly different
spaces $Y_0,\ldots,Y_{N-1}$ can be treated analogously).
In practical situations, we do not know the data exactly. Instead,
we have only approximate measured data $y_i^\delta \in Y$ satisfying
\begin{equation}\label{eq:noisy-i}
    \norm{ y_i^\delta - y_i } \le \delta_i \, , \ \ i = 0, \dots, N-1 \, ,
\end{equation}
with $\delta_i > 0$ (noise level). We use the notation $\delta :=
(\delta_0, \dots, \delta_{N-1})$. The finite set of data above is obtained
by indirect measurements of the parameter $x$, this process being described by
the model
\begin{equation}\label{eq:inv-probl}
    F_{i}(x)  =  y_{i} \, , \ \ i = 0, \dots, N-1 \, ,
\end{equation}
where $F_i: D_i \subset X \to Y$, and $D_i$ are the corresponding
domains of definition.

Standard methods for the solution of system \eqref{eq:inv-probl} are
based in the use of \textit{Iterative type} regularization methods
\cite{BakKok04, EngHanNeu96, HanNeuSch95, KalNeuSch08, Lan51} or
\textit{Tikhonov type} regularization methods \cite{EngHanNeu96,
Mor93, SeiVog89, Tik63b, TikArs77, Sch93a} after rewriting
(\ref{eq:inv-probl}) as a single equation $\F(x) = y$, where
\begin{align} \label{eq:single-op}
\F \ : \ = ( F_0, \dots, F_{N-1} ): \bigcap\nolimits_{i=0}^{N-1} D_i \to Y^N
\end{align}
and $y := (y_0, \dots,  y_{N-1})$.

The starting point of our approach is the Levenberg-Marquardt method
\cite{Ha97, Le44, Ma63, DES98} for solving ill-posed problems, which is defined by
$$
  x_{k+1}^\delta \ = \ x_k^\delta - ( \F'(x_k^\delta)^* \F'(x_k^\delta)
  +  \alpha I)^{-1} \F'(x_k^\delta)^* ( \F(x_k^\delta) -  y^\delta ) \, .
$$
where $\F'(z)$ is the Frechet-derivative of $\F$ in $z$ and $\F'(z)^*$ is its
adjoint.
Motivated by the ideas in \cite{HLS07,CHLS08, HLR09, BK06}, we propose
in this article a \textit{loping Levenberg-Marquardt-Kaczmarz method}
(\textsc{l-LMK} method) for solving (\ref{eq:inv-probl}). This
iterative method is defined by
\begin{equation}  \label{eq:lmk}
x_{k+1}^\delta \ = \ x_{k}^\delta + \omega_k \, h_k \, ,
\end{equation}
where
\begin{equation} \label{def:hk}
h_k \ := \ ( F_{[k]}'(x_k^\delta)^* F_{[k]}'(x_k^\delta) + \alpha I)^{-1}
          F_{[k]}'(x_k^\delta)^* ( y_{[k]}^\delta - F_{[k]}(x_k^\delta) )
\end{equation}
and
\begin{equation} \label{def:omk}
\omega_k \ := \
  \begin{cases}
    1  & {\rm if}\ \norm{F_{[k]}(x_{k}^\delta) - y_{[k]}^\delta} \geq \tau \delta_{[k]} \\
    0  & \text{otherwise}
  \end{cases} \, .
\end{equation}
Here $\alpha > 0$ is an appropriately chosen number (see \eqref{def:alp-tau}
below), $[k] := (k \mod N) \in \set{0, \dots, N-1}$, and $x_0^\delta = x_0
\in X$ is an initial guess, possibly incorporating some \textit{a priori}
knowledge about the exact solution, and $\tau>1$ a fixed constant
(see \eqref{def:alp-tau} below).

The \textsc{l-LMK} method consists in incorporating the Kaczmarz
strategy into the Levenberg-Marquardt method. This procedure is analog
to the one introduced in \cite{HLS07}, \cite{CHLS08}, \cite{HLR09}, and
\cite{BK06} regarding the Landweber-Kaczmarz (\textsc{LK}) iteration,
the Steepest-Descent-Kaczmarz (\textsc{SDK}) iteration, the
Expectation-Maximization-Kaczmarz (\textsc{EMK}) iteration, and the
Iteratively Regularized Gauss-Newton-Kaczmarz (\textsc{IRGNK}) respectively.
As usual in Kaczmarz type algorithms, a group of $N$ subsequent steps
(starting at some multiple $k$ of $N$) is called a {\em cycle}. The
\textsc{l-LMK} iteration should be terminated when, for the first time, all
$x_k^\delta$ are equal within a cycle. That is, we stop the iteration at
\begin{equation} \label{def:discrep-lmk}
  k_*^\delta \ := \ \min \{ lN \in \N: \, x_{lN}^\delta =
                 x_{lN+1}^\delta = \cdots = x_{lN+N}^\delta \} \, ,
\end{equation}
Notice that $k_*^\delta$ is the smallest multiple of $N$ such that
\begin{equation} \label{def:discr-lmk2}
x_{k_*^\delta}^\delta = x_{k_*^\delta+1}^\delta = \dots = x_{k_*^\delta+N}^\delta \, ,
\end{equation}
or equivalently (see Proposition \ref{prop:st-fin-lmk} below) such that
$$
\omega_{k_*^\delta-1} = \omega_{k_*^\delta} = \dots = \omega_{k_*^\delta+N-1} = 0 \, .
$$
For exact data ($\delta = 0$) we have $\omega_k = 1$ for each $k$ and the
\textsc{l-LMK} iteration reduces to the Levenberg-Marquardt-Kaczmarz
(\textsc{LMK}) method. For noisy data however, the \textsc{l-LMK} method is
fundamentally different from the \textsc{LMK} method: The bang-bang
relaxation parameter $\omega_k$ effects that the iterates defined in
\eqref{eq:lmk}, \eqref{def:hk} become stationary if all components of the
residual vector  $\norm{ F_i(x_k^\delta) - y_i^{\delta} }$  fall below a
pre-specified threshold.
This characteristic renders \eqref{eq:lmk}, \eqref{def:hk} a regularization
method in the sense of \cite{EngHanNeu96} (see Section~\ref{sec:conv-noise}).
\medskip

The article is outlined as follows.
In Section \ref{sec:basic} we formulate basic assumptions and derive
some auxiliary estimates required for the analysis.
In Section \ref{sec:conv-exact} we prove a convergence result for the
\textsc{LMK} method.
In Section \ref{sec:conv-noise} we prove a semiconvergence result for
the \textsc{l-LMK} method.
In Section~\ref{sec:numeric} a numerical experiment for an inverse doping
problem is presented.
Section~\ref{sec:conclusion} is devoted to final remarks and conclusions.

% --------------------------------------------------------------------
\section{Assumptions and basic results} \label{sec:basic}

We begin this section by introducing some assumptions, that are needed
for the convergence analysis presented in the next sections. These
assumptions derive from the classical assumptions used in the analysis of
iterative regularization methods \cite{EngHanNeu96, KalNeuSch08, Sch93a}.
\medskip

\noindent (A1) \
The operators $F_i$ and their linearizations $F_i'$ -- see (A2) -- are continuous and
the corresponding domains of definition $D_i$ have nonempty interior, i.e., there exists
$x_0 \in X$, $\rho>0$ such that $B_\rho(x_0)\subset \bigcap\nolimits_{i=0}^{N-1} D_i$, where $B_\rho(x_0)$ is the ball of radius $\rho$ around $x_0$.
%
% The operators $F_i$ are weakly sequentially continuously and Fr\'echet
% differentiable and the corresponding domains of definition $D_i$ are
% weakly closed.
Moreover, we assume the existence of $C > 0$ such that
\begin{equation} \label{eq:a-dfb}
\| F_i'(x) \| \ \le \ C \, , \quad \ x \in B_\rho(x_0)
\end{equation}
(notice that $x_0^\delta =x_0$ is used as starting point of the \textsc{l-LMK} iteration).
\medskip

\noindent (A2) \
We assume that the {\em local tangential cone condition} \cite{EngHanNeu96, KalNeuSch08}
\begin{equation} \label{eq:a-tcc}
\| F_i(\bar{x}) - F_i(x) -  F_i'(x)( \bar{x} - x ) \|_Y \ \leq \
   \eta \norm{ F_i(\bar{x}) - F_i(x) }_{Y} \, , \qquad
   \forall \ x, \bar{x} \in B_{\rho}(x_0)
\end{equation}
holds for some $\eta < 1$. This is a uniform assumption on the nonlinearity
of the operators $F_i$.
Note that $F_i'(x)$ need not necessarily be the Fr\'echet derivative of
$F_i$ at $x$, but it should be a bounded linear operator that continuously
depends on $x$, see (A1).
\medskip

\noindent (A3) \
There exists an element $x^* \in B_{\rho/2}(x_0)$ such that $\F(x^*) = y$,
where $y = (y_0, \dots,  y_{N-1})$ are the exact data satisfying
\eqref{eq:noisy-i}.
\medskip

We are now in position to choose the positive constants $\alpha$
and $\tau$ in \eqref{def:hk}, \eqref{def:omk}. For the rest of
this article we shall assume
\begin{equation} \label{def:alp-tau}
\alpha \ > \ \frac{C^2 q}{1 - q} \, ,
\qquad
\tau \ > \ \frac{1+\eta}{1-\eta} > 1 \, ,
\qquad
\eta + (1+\eta) \tau^{-1} < q < 1\, 
\end{equation}
for some $0 < q < 1$.

In the sequel we verify some basic facts that are helpful for
the convergence analysis derived in the next two sections. The
first result concerns some useful identities.

\begin{lemma} \label{lem:ident}
Let $x_k^\delta$, $h_k$ and $\alpha$ be defined by \eqref{eq:lmk},
\eqref{def:hk} and \eqref{def:alp-tau} respectively. Moreover,
assume that (A1) - (A3) hold true.

\noindent a) \ For all $k \in \N$ we have
$$
y_{[k]}^\delta - F_{[k]}(x_k^\delta) - F_{[k]}'(x_k^\delta) h_k
\ = \
\alpha ( F_{[k]}'(x_k^\delta) F_{[k]}'(x_k^\delta)^* + \alpha I)^{-1}
       ( y_{[k]}^\delta - F_{[k]}(x_k^\delta) ) \, .
$$
b) \ Moreover, if $\omega_k = 1$, we have
$$
h_k \ = \ -\alpha^{-1} F_{[k]}'(x_k^\delta)^* \big[
        F_{[k]}'(x_k^\delta) ( x_{k+1}^\delta - x_k^\delta )
        + F_{[k]}(x_k^\delta) - y_{[k]}^\delta \big] \, ,
$$
$$
F_{[k]}'(x_k^\delta) ( x_{k+1}^\delta - x_k^\delta )
                     + F_{[k]}(x_k^\delta) - y_{[k]}^\delta \ = \
\alpha ( F_{[k]}'(x_k^\delta) F_{[k]}'(x_k^\delta)^* + \alpha I)^{-1}
                     ( F_{[k]}(x_k^\delta) - y_{[k]}^\delta ) \, .
$$
c) \ Define \ $B_k^\delta :=
\alpha ( A_k A_k^* + \alpha I)^{-1} ( F_{[k]}(x_k^\delta) - y_{[k]}^\delta )
= F_{[k]}'(x_k^\delta) ( x_{k+1}^\delta - x_k^\delta ) +
F_{[k]}(x_k^\delta) - y_{[k]}^\delta$ (we write $B_k=B_k^\delta$ for $\delta=0$).
Then
\begin{equation} \label{eq:monot-auxB}
q \, \norm{ F_{[k]}(x_k^\delta) - y_{[k]}^\delta } \ \le \
\norm{ B_k^\delta } \ \le \
\norm{ F_{[k]}(x_k^\delta) - y_{[k]}^\delta } \, .
\end{equation}
\end{lemma}
\begin{proof}
The proof of a) and b) is straightforward and will be omitted.
To prove c) notice that
$$
%  \norm{ B_k^\delta } =
   \norm{ \alpha ( A_k A_k^* + \alpha I)^{-1}
   ( F_{[k]}(x_k^\delta) - y_{[k]}^\delta ) }
\geq
  \frac{\alpha}{C^2 + \alpha}
  \norm{ F_{[k]}(x_k^\delta) - y_{[k]}^\delta }
\geq q \, \norm{ F_{[k]}(x_k^\delta) - y_{[k]}^\delta } \, 
$$
with $A_k := F_{[k]}'(x_k^\delta)$. On the other hand, we have
$\norm{ \alpha ( A_k A_k^* + \alpha I)^{-1}
( F_{[k]}(x_k^\delta) - y_{[k]}^\delta ) }
\leq \norm{ F_{[k]}(x_k^\delta) - y_{[k]}^\delta }$.
\end{proof}

\begin{remark}
According to \cite{Ha97} the Levenberg-Marquardt-iteration should be
implemented with variable $\alpha_k$, which for the \textsc{LMK} would
mean that $\alpha_k$ is chosen in such a way that $h_k = h_k(\alpha)$ in
\eqref{def:hk} satisfies
\begin{equation}  \label{alphakHanke}
\norm{ F_{[k]}'(x_k^\delta) h_k + F_{[k]}(x_k^\delta) - y_k^\delta }
\ = \ q \, \norm{ F_{[k]}(x_k^\delta) - y_k^\delta } \, ,
\end{equation}
for some $0 < q < 1$. From \eqref{eq:monot-auxB} and monotonicity of the
mapping $\alpha\mapsto \norm{ B_k^\delta }$, (see, e.g., \cite{EngHanNeu96}),
it follows that $\alpha$ as chosen in \eqref{def:alp-tau} is larger than the
$\alpha_k$'s defined in \cite{Ha97} (see \cite[Theorem~3.3.1]{Gr84}).

It is worth noticing that Lemma~\ref{lem:monot-aux} as well as 
Proposition~\ref{prop:monot} remain valid with $\alpha_k$ chosen
as in \eqref{alphakHanke}.

\end{remark}

The following lemma is an important auxiliary result, which will be used
to prove a monotonicity property of the \textsc{l-LMK} iteration.

\begin{lemma} \label{lem:monot-aux}
Let $x_k^\delta$, $h_k$, $\alpha$ and $q$ be defined by \eqref{eq:lmk},
\eqref{def:hk} and \eqref{def:alp-tau} respectively. Moreover, assume that
(A1) - (A3) hold true. If $x_k^\delta \in B_\rho(x_0)$ for some $k \in \N$, then
\begin{multline} \label{eq:monot-aux}
\| x_{k+1}^\delta  - x^* \|^2 - \| x_k^\delta - x^* \|^2  \le
   2 \frac{\omega_k}{\alpha} \norm{ B_k^\delta } \,
   \big[ (\eta-q) \norm{ F_{[k]}(x_k^\delta) - y_{[k]}^\delta }
         + (1+\eta) \delta_{[k]} \big] \\
    - \norm{ x_{k+1}^\delta - x_k^\delta}^2 \, ,
\end{multline}
where $B_k^\delta$ is defined as in Lemma~\ref{lem:ident}.
\end{lemma}
\begin{proof}
Let $A_k := F'_{[k]}(x_k^\delta)$. If $\omega_k = 0$,
\eqref{eq:monot-aux} is obvious. 
If $\omega_k = 1$, it follows from \eqref{eq:lmk}, \eqref{def:hk}
and Lemma~\ref{lem:ident} that
\begin{align*}
\norm{ & x_{k+1}^\delta - x^*}^2 - \norm{x_k^\delta - x^*}^2
\\
&\quad = 2\, \langle x_{k+1}^\delta - x_k^\delta , \ x_{k+1}^\delta - x^* \rangle
         - \norm{ x_{k+1}^\delta - x_k^\delta}^2
\\
&\quad = -2 \alpha^{-1}\, \langle A_k ( x_{k+1}^\delta - x_k^\delta )
           + F_{[k]}(x_k^\delta) - y_{[k]}^\delta , \
           A_k ( x_{k+1}^\delta - x^* ) \pm A_k x_k^\delta
           \pm F_{[k]}(x_k^\delta) \pm  y_{[k]}^\delta \rangle \\
&\quad\quad - \norm{ x_{k+1}^\delta - x_k^\delta}^2
\\
& \quad = -2 \alpha^{-1}\, \big[ \norm{ B_k^\delta }^2 +
          \langle B_k^\delta , \ A_k ( x_k^\delta - x^* )
          - F_{[k]}(x_k^\delta) + y_{[k]}^\delta \pm F_{[k]}(x^*) \rangle
          - \norm{ x_{k+1}^\delta - x_k^\delta}^2
\\
& \quad = 2 \alpha^{-1}\, \big[ - \norm{ B_k^\delta }^2 +
          \langle B_k^\delta , \ - F_{[k]}(x^*) + F_{[k]}(x_k^\delta)
          + A_k ( x^* - x_k^\delta ) \rangle +
          \langle B_k^\delta , \ F_{[k]}(x^*) - y_{[k]}^\delta \rangle \\
&\quad\quad - \norm{ x_{k+1}^\delta - x_k^\delta}^2 \, .
\end{align*}
Now, applying the Cauchy-Schwarz inequality, \eqref{eq:noisy-i}
and \eqref{eq:a-tcc} with $\bar x = x^* \in B_{\rho/2}(x_0)$,
$x = x_k^\delta \in B_\rho(x_0)$, leads to
\begin{align} \label{eq:monot-auxA}
\norm{ & x_{k+1}^\delta - x^*}^2 - \norm{x_k^\delta - x^*}^2
\nonumber \\
&\quad \le 2 \alpha^{-1}\, \big[ - \norm{ B_k^\delta }^2 +
           \norm{ B_k^\delta } \ \eta \norm{ F_{[k]}(x_k^\delta) - F_{[k]}(x^*)
           \pm y_{[k]}^\delta } + \norm{ B_k^\delta } \delta_{[k]} \big] \nonumber
           - \norm{ x_{k+1}^\delta - x_k^\delta}^2
\\
&\quad \le 2 \alpha^{-1} \norm{ B_k^\delta } \, \big[ - \norm{ B_k^\delta } +
           \eta \norm{ F_{[k]}(x_k^\delta) - y_{[k]}^\delta }
           + (\eta+1) \delta_{[k]} \big]
           - \norm{ x_{k+1}^\delta - x_k^\delta}^2 \, .
\end{align}
The estimate \eqref{eq:monot-aux} follows now plugging \eqref{eq:monot-auxB}
into \eqref{eq:monot-auxA}.
\end{proof}

Our next goal is to prove a monotonicity property, known to be satisfied
by classical iterative regularization methods (e.g., Landweber
\cite{EngHanNeu96}, steepest descent \cite{Sch96}), and also by
Kaczmarz type methods (e.g., loping Landweber-Kaczmarz \cite{HLS07},
loping Steepest-Descent-Kaczmarz \cite{CHLS08}, loping Expectation-%
Maximization-Kaczmarz \cite{HLR09}).

\begin{propo}[Monotonicity] \label{prop:monot}
Under the assumptions of Lemma~\ref{lem:monot-aux}, for all $k < k_*^\delta$
the iterates $x_k^\delta$ remain in $B_{\rho/2}(x^*) \subset B_{\rho}(x_0)$
and satisfy \eqref{eq:monot-aux}. Moreover,
\begin{equation} \label{eq:lmk-monot}
\| x_{k+1}^\delta - x^* \|^2 \ \le \ \| x_k^\delta - x^* \|^2 \, ,
\quad  k < k_*^\delta \, .
\end{equation}
\end{propo}
\begin{proof}
From (A3) it follows that $x_0 \in B_{\rho/2}(x^*)$.
If $\omega_0 = 0$, then $x_1^\delta = x_0^\delta = x_0 \in
B_{\rho/2}(x^*)$ and \eqref{eq:lmk-monot} is satisfied with
equality for $k=0$.
If $\omega_0 = 1$, it follows from Lemma~\ref{lem:monot-aux} that
\eqref{eq:monot-aux} holds for $k=0$. Then we conclude from
\eqref{def:omk} and \eqref{eq:monot-aux} that
$$
\| x_{k+1}^\delta  - x^* \|^2 - \| x_k^\delta - x^* \|^2  \le
   2 \alpha^{-1} \norm{ B_k^\delta } \norm{ F_{[k]}(x_k^\delta) - y_{[k]}^\delta }
   \big[ \eta + (1+\eta) \tau^{-1} - q \big] \, .
$$
Due to \eqref{def:alp-tau} the last term on the right hand side is
non positive. Thus, \eqref{eq:lmk-monot} holds for $k=0$.
In particular we have $x_1^\delta \in B_{\rho/2}(x^*)$. The proof follows
now using an inductive argument.
\end{proof}

In the next two sections we provide a complete convergence
analysis for the \textsc{l-LMK} iteration, showing that it is a
convergent regularization method in the sense of \cite{EngHanNeu96}.

% --------------------------------------------------------------------
\section{Convergence for exact data} \label{sec:conv-exact}

Unless otherwise stated, we assume in the sequel that (A1) - (A3) hold
true and that $x_k^\delta$, $h_k$, $\alpha$, $\tau$ and $q$ are defined
by \eqref{eq:lmk}, \eqref{def:hk} and \eqref{def:alp-tau}.
Our main goal in this section is to prove convergence in the case
$\delta_i = 0$, $i=0,\dots,N-1$.
As already observed in Section~\ref{sec:intro} the \textsc{l-LMK}
reduces in this case to the \textsc{LMK} iteration (i.e. $\omega_k = 1$
in \eqref{def:omk}).
For exact data $y = (y_0, \dots, y_{N-1})$, the iterates in \eqref{eq:lmk}
are denoted by $x_k$, in contrast to $x_k^\delta$ in the noisy data case.

%JB1305
%  Die folgende Bemerkung sollte besser erst vor dem nächsten Theorem
% positioniert werden. Hier ist sie ohne Beziehung.
%JB1305

\begin{remark} \label{rem:mns}
It is worth noticing that there exists an $x_0$-minimal norm solution of
\eqref{eq:inv-probl} in $B_{\rho/2}(x_0)$, i.e., a solution $x^\dag$ of
\eqref{eq:inv-probl} such that $\norm{ x^\dag - x_0 } =
\inf \{ \norm{x - x_0} : x \in B_{\rho/2}(x_0)$ and $\F(x) = y \}$.
Moreover, $x^\dag $ is the only solution of \eqref{eq:inv-probl} in
$B_{\rho/2}(x_0) \cap \big( x_0 + \ker(F '( x^\dag ))^\perp \big) $.
This assertion is a direct consequence of \cite[Proposition~2.1]{HanNeuSch95}.
For a detailed proof we refer the reader to \cite{KalNeuSch08}.
\end{remark}

In the sequel we derive some estimates that are helpful for the proof of
the convergence result. From Proposition~\ref{prop:monot} it follows that
\eqref{eq:monot-aux} holds for all $k \in \N$. Since the data is exact,
\eqref{eq:monot-aux} can be rewritten as
\begin{equation} \label{eq:bound-ser-aux}
\| x_{k+1} - x^* \|^2 - \| x_k - x^* \|^2 \ \le \
2 \alpha^{-1} (\eta-q) \norm{ B_k } \, \norm{ F_{[k]}(x_k) - y_{[k]}}
- \norm{ x_{k+1} - x_k}^2 \, .
\end{equation}
Now inserting \eqref{eq:monot-auxB} into \eqref{eq:bound-ser-aux}
and summing over all $k$, leads to
\begin{subequations}
\begin{equation} \label{eq:bound-ser1}
\summ_{k=0}^\infty \norm{ F_{[k]}(x_k) - y_{[k]} }^2 \ \leq \
      \alpha [2q (q-\eta)]^{-1} \norm{x_0 - x^*}^2 \ < \ \infty
\end{equation}
(notice that $q > \eta + (1+\eta) \tau^{-1} > \eta$), and also to
\begin{equation} \label{eq:bound-ser4}
\summ_{k=0}^\infty \norm{ B_k }^2 \ \leq \
      \alpha [2 (q-\eta)]^{-1} \norm{x_0 - x^*}^2 \ < \ \infty \, .
\end{equation}
On the other hand, neglecting the last term on the right hand side
of \eqref{eq:bound-ser-aux} and summing over all $k$, leads to
\begin{equation} \label{eq:bound-ser2}
\summ_{k=0}^\infty \norm{ B_k } \, \norm{ F_{[k]}(x_k) - y_{[k]} } \ \leq \
      \alpha [2 (q-\eta)]^{-1} \norm{x_0 - x^*}^2 \ < \ \infty \, .
\end{equation}
Finally, neglecting the first term on the right hand side of
\eqref{eq:bound-ser-aux} and summing over all $k$, leads to
\begin{equation} \label{eq:bound-ser3}
\summ_{k=0}^\infty \norm{ x_{k+1} - x_k}^2 \ \leq \
      \norm{x_0 - x^*}^2 \ < \ \infty \, .
\end{equation}
\end{subequations}

\begin{theorem}[Convergence for exact data] \label{th:exact}
For exact data, the iteration $x_k$ converges to a solution of
\eqref{eq:inv-probl}, as $k \to \infty$. Moreover, if the kernel
condition \cite{DES98}
\begin{equation} \label{eq:kern-cond}
\nr( \F'(x^\dag) ) \subseteq \nr( \F'(x) )
\quad \text{ for all } x \in B_\rho(x_0) \, ,
\end{equation}
is satisfied, where $\F$ is defined as in \eqref{eq:single-op},
then $x_k \to x^\dag$.
\end{theorem}
\begin{proof}
We define $e_k := x^\dag - x_k$. From Proposition \ref{prop:monot} it
follows that $\norm{e_k}$ is monotone non-increasing. Therefore,
$\norm{e_k}$ converges to some $\epsilon \geq 0$. In the following
we show that $e_k$ is in fact a Cauchy sequence.

In order to show that $e_k$ is a Cauchy sequence, it suffices to
prove $|\langle e_n - e_k , e_n \rangle | \to 0$,
$|\langle e_n - e_l , e_n \rangle | \to 0$ as $k , l \to \infty$
with $k \leq l$ for some $k \leq n \leq l$ \cite[Theorem~2.3]{HanNeuSch95}.
Let $k \leq l$ be arbitrary, $k = k_0 N + k_1$, $l = l_0 N + l_1$,
$k_1, l_1 \in \{ 0, \ldots, N-1\}$, let $n_0 \in \{ k_0, \ldots, l_0\}$
be such that
\begin{multline} \label{n0min}
\summ_{s=0}^{N-1}
\big\{ \| x_{n_0 N+s+1} - x_{n_0 N+s} \| + \| F_{s}(x_{n_0 N+s}) - y_s \| \big\}
\leq \\ \leq
\summ_{s=0}^{N-1}
\big\{ \| x_{i_0 N+s+1} - x_{i_0 N+s} \| + \|F_{s}(x_{i_0 N+s}) - y_s \| \big\}
\, , \quad \mbox{for all } i_0 \in \{ k_0, \ldots, l_0 \} \, ,
\end{multline}
and set $n = n_0 N + N-1$. Therefore
\begin{eqnarray}
\lefteqn{ |\langle e_n - e_k , e_n \rangle | } \nonumber \\
&=&
\Big| \summ_{i=k}^{n-1}
\langle (x_{i+1} - x_i), (x_n - x^\dag) \rangle \Big| \nonumber \\
&=&
\Big| \summ_{i=k}^{n-1} \alpha^{-1}
\langle F_{[i]}'(x_i)(x_{i+1}-x_i)+F_{[i]}(x_i)-y_{[i]} ,
        F_{[i]}'(x_i)(x_n - x^\dag) \rangle \Big| \nonumber \\
&=&
\Big| \summ_{i=k}^{n-1} \alpha^{-1}
\langle F_{[i]}'(x_i)(x_{i+1} - x_i) + F_{[i]}(x_i) - y_{[i]} ,
        F_{[i]}'(x_i)(x_n - x_i) + F_{[i]}'(x_i) (x_i - x^\dag)
\rangle \Big| \nonumber \\
&\leq&
\summ_{i=k}^{n-1} \alpha^{-1}
% \| F_{[i]}'(x_i)(x_{i+1} - x_i) + F_{[i]}(x_i) - y_{[i]} \| \nonumber \\
% &&\qquad\qquad 
\| B_i \| \,
\big[ (1+\eta) \, \| F_{[i]}(x_n) - F_{[i]}(x_i) \|
      +(1+\eta) \, \|F_{[i]}(x_i) - F_{[i]}(x^\dag) \| \big] \nonumber \\
&\leq&
\alpha^{-1} (1+\eta) \summ_{i=k}^{n-1}
% \| F_{[i]}'(x_i)(x_{i+1} - x_i) + F_{[i]}(x_i) - y_{[i]} \|
\| B_i \| \,
\big[ \| F_{[i]}(x_n) - y_{[i]} \| + 2 \|F_{[i]}(x_i) - y_{[i]} \| \big]
\label{esteneken}
\end{eqnarray}
With $i=i_0 N+i_1$ we get
\begin{eqnarray*}
\lefteqn{ \|F_{[i]}(x_n) - y_{[i]} \| } \\
&=& \| F_{i_1}(x_{n_0 N+N-1}) - y_{i_1} \| \\
&\leq& \| F_{i_1}(x_{n_0 N + i_1}) - y_{i_1} \|
+ \summ_{s=i_1}^{N-2} \|F_{i_1}(x_{n_0 N+s+1}) - F_{i_1}(x_{n_0 N+s}) \| \\
&\leq& \| F_{i_1}(x_{n_0 N+i_1}) - y_{i_1} \|
+ \frac{1}{1-\eta} \summ_{s=i_1}^{N-2} \|F_{i_1}'(x_{n_0 N+s})
  (x_{n_0 N+s+1} - x_{n_0 N+s}) \| \\
&\leq& \| F_{i_1}(x_{n_0 N+i_1}) - y_{i_1} \|
 + \frac{C}{1-\eta} \summ_{s=i_1}^{N-2} \| x_{n_0 N+s+1} - x_{n_0 N+s} \| \\
&\leq& \Big( 1+\frac{C}{1-\eta} \Big) \summ_{s=i_1}^{N-2}
\big\{ \| x_{n_0 N+s+1} - x_{n_0 N+s} \| + \| F_{s}(x_{n_0 N+s}) - y_s \| \big\}
\, .
\end{eqnarray*}
Hence, by minimality (\ref{n0min}) we get
$$
\| F_{[i]}(x_n)-y_{[i]}\| \ \leq \
\Big( 1 + \frac{C}{1-\eta} \big) \summ_{s=0}^{N-1}
\big\{ \|x_{i_0 N+s+1} - x_{i_0 N+s} \|
 + \| F_{s}(x_{i_0 N+s}) - y_s \| \big\} \, .
$$
Inserting this into (\ref{esteneken}) we obtain
\begin{eqnarray} \label{esteneken1}
| \langle e_n - e_k , e_n \rangle |
\ \leq \
2 \alpha^{-1} (1+\eta) \summ_{i=k}^{n-1}
% \| F_{[i]}'(x_i)(x_{i+1} - x_i) + F_{[i]}(x_i) - y_{[i]}\|
\| B_i \| \, \| F_{[i]}(x_i) - y_{[i]}\| \displaystyle
 + \Big( 1 + \frac{C}{1-\eta} \Big) (1+\eta) \ \mbox{sum}_2
\end{eqnarray}
with 
\begin{eqnarray*}
\mbox{sum}_2 &=& \summ_{i_0=k_0}^{n_0-1} \summ_{i_1=0}^{N-1}
\alpha^{-1} \| F_{i_1}'(x_{i_0 N+i_1}) (x_{i_0 N+i_1+1} - x_{i_0 N+i_1})
 + F_{i_1}(x_{i_0 N+i_1}) - y_{i_1} \| \\
&&\qquad\quad
\summ_{s=0}^{N-1} \big\{
\| x_{i_0 N+s+1} - x_{i_0 N+s}\| + \|F_{s}(x_{i_0 N+s}) - y_s \| \big\} \\
&\leq&
N \summ_{i_0=k_0}^{n_0-1} \summ_{i_1=0}^{N-1} \displaystyle
\Big\{ \frac12 \Big( \alpha^{-1}
\| F_{i_1}'(x_{i_0 N+i_1}) (x_{i_0 N+i_1+1} - x_{i_0 N+i_1})
 + F_{i_1}(x_{i_0 N+i_1}) - y_{i_1} \| \Big)^2 \\
&& \qquad \qquad \quad + \ \| x_{i_0 N+i_1+1} - x_{i_0 N+i_1}\|^2
 + \| F_{i_1}(x_{i_0 N+i_1}) - y_{i_1}\|^2 \Big\} \\
&\leq&
\frac{N}{\min\{2\alpha,1\}} \summ_{i=k_0}^{n}
\Big\{ \alpha^{-1}
% \| F_{[i]}'(x_{i}) (x_{i+1} - x_{i}) + F_{[i]}(x_{i}) - y_{[i]}\|^2 \\
\| B_i \|^2 + \| x_{i_0 N+i_1+1} - x_{i_0 N+i_1} \|^2
 + \| F_{i_1}(x_{i_0 N+i_1}) - y_{i_1} \|^2 \Big\} \, .
\end{eqnarray*}
Hence by \eqref{eq:bound-ser1}, \eqref{eq:bound-ser2}, \eqref{eq:bound-ser3},
\eqref{eq:bound-ser4}, both terms on the right hand side of (\ref{esteneken1})
go to zero as $k,l\to\infty$.
Analogously one shows that $\bracket{e_n - e_l , e_n} \to 0$ as $l \to \infty$.

Thus, $e_k$ is a Cauchy sequence and $x_k = x^\dag - e_k$ converges to some
% element $x^* \in X$. Since the residuals $\norm{ F_{[k]}(x_k) - y_{[k]} }$
% converge to zero, $x^*$ is solution of \eqref{eq:inv-probl}.
% JB1305
element $x^+ \in X$. Since the residuals $\norm{ F_{[k]}(x_k) - y_{[k]} }$
converge to zero, $x^+$ is a solution of \eqref{eq:inv-probl}.
% x^* hat schon eine Bedeutung in A3). Oder sehe ich dies falsch?
% JB1305
\medskip

Now assume $\nr( \F'(x^\dag) ) \subseteq \nr( \F'(x) )$, for $x \in
B_\rho(x_0)$. Since $x_{k+1} - x_k$ is either zero or $h_k$,
it follows from Lemma~\ref{lem:ident}~b) that $x_{k+1} - x_k \in
\ra(F_{[k]}'(x_k)^*) \subset \nr( F_{[k]}'(x_k) )^\bot \subset
\nr( \F'(x_k) )^\bot \subset \nr( \F'(x^\dag) )^\bot$.
An inductive argument shows that all iterates $x_k$ are elements
of $x_0 + \nr( \F'(x^\dag) )^\bot$.
% Together with the  continuity of $\F'(x^\dag)$, this implies that
Therefore, $x^* \in x_0 + \nr( \F'(x^\dag))^\bot$.
By Remark~\ref{rem:mns}, $x^\dag$ is the only solution of \eqref{eq:inv-probl}
in $B_{\rho/2}(x_0) \cap ( x_0  + \nr(\F '( x^\dag ))^\perp )$, and
so the second assertion follows.
\end{proof}

\begin{remark}
In order to consider the variable choice of $\alpha$ according to \eqref{alphakHanke}
let us consider for the moment a condition which is slightly
stronger than the tangential cone condition, namely the
{\em range invariance condition}
\medskip

\noindent (A2') \
There exist linear bounded operators $R_i(\bar{x},x)$ satisfying
\begin{equation} \label{eq:a-stcc}
F_i'(\bar{x}) \ = \ R_i(\bar{x},x) \, F_i'(x) \ , \quad
\norm{ R_i(\bar{x},x) - I } \ \le \ c_R \, , \qquad
x, \bar{x} \in B_{\rho}(x_0) \, ,
\end{equation}
for some $0 < c_R < 1$. \\
Notice that from $F_i'(x)( \bar{x} - x ) =
\big( \int_0^1 R_i(x + \theta (\bar{x} - x) , x) \, d\theta
\big)^{-1} ( F_i(\bar{x}) - F_i(x) )$, it follows that (A2')
implies (A2) with $\eta = c_R (1-c_R)^{-1}$.

If (A2) is substituted by (A2') in Theorem~\ref{th:exact},
the estimates can be improved to
\begin{eqnarray}
\lefteqn{|\langle e_n-e_k,e_n\rangle |}\nonumber\\
&\leq&
\frac{c_R}{2(1-c_R)}
\summ_{i_0=k_0}^{n_0-1}
\summ_{i_1=0}^{N-1} \left\{
\frac{\omega_{i_0 N+i_1}}{\alpha_{i_0 N+i_1}}
\|F_{i_1}'(x_i)(x_{i_0 N+i_1+1}-x_{i_0 N+i_1}
+F_{i_1}(x_{i_0 N+i_1})-y_{i_1}\|^2\hspace*{1.5cm}\right.\nonumber\\
&&\qquad\qquad\qquad
\left. + {\textstyle \frac{1}{q^2}\frac{\omega_{i_0 N+i_1}}{\alpha_{i_0 N+i_1}} }
\|F_{i_1}'(x_i)(x_{n_0 N+i_1+1}-x_{n_0 N+i_1})
+F_{i_1}(x_{n_0 N+i_1})-y_{i_1}\|^2
\right\}
\label{estimproivedk}\\
&&+\frac{N}{2q^2} \summ_{i_0=k_0}^{n_0-1}
\summ_{i_1=0}^{N-1}\left\{
\frac{\omega_{i_0 N+i_1}}{\alpha_{i_0 N+i_1}}
\|F_{i_1}'(x_i)(x_{i_0 N+i_1+1}-x_{i_0 N+i_1}
+F_{i_1}(x_{i_0 N+i_1})-y_{i_1}\|^2\right.\nonumber\\
&&\qquad\qquad\qquad
\left. + {\textstyle \frac{\omega_{n_0 N+i_1}}{\alpha_{n_0 N+i_1}} }
\|F_{i_1}'(x_i)(x_{n_0 N+i_1+1}-x_{n_0 N+i_1})
+F_{i_1}(x_{n_0 N+i_1})-y_{i_1}\|^2
\right\} \, .
\nonumber
\end{eqnarray}
This suggests that if, instead of \eqref{n0min} the index $n_0$ is
chosen from the more natural requirement
\begin{multline*} %\label{n0minimproved}
\summ_{s=0}^{N-1} \frac{\omega_{n_0 N+s}}{\alpha_{n_0 N+s}}
\| F_{s}'(x_{n_0 N+s}) ( x_{n_0 N+s+1} - x_{n_0 N+s} ) + F_{s}(x_{n_0 N+s}) - y_s \|^2
\leq \\
\summ_{s=0}^{N-1} \frac{\omega_{i_0 N+s}}{\alpha_{i_0 N+s}}
\| F_{s}'(x_{i_0 N+s}) ( x_{i_0 N+s+1} - x_{i_0 N+s} ) + F_{s}(x_{i_0 N+s}) - y_s \|^2
\quad \mbox{for all } i_0 \in \{ k_0, \ldots, l_0 \} \, ,
\end{multline*}
then the analysis might extend to the variable parameter choice \eqref{alphakHanke}.
Note however, that due to the factor $\frac{\omega_{i_0 N+i_1}}{\alpha_{i_0 N+i_1}}$
(instead of an in this context desirable $i_0$-independent factor
$\frac{\omega_{n_0 N+i_1}}{\alpha_{n_0 N+i_1}}$) in (\ref{estimproivedk}),
this is unfortunately not the case.
\end{remark}

% --------------------------------------------------------------------
\section{Convergence for noisy data} \label{sec:conv-noise}

Throughout this section, we assume that (A1) - (A3) hold true
and that $x_k^\delta$, $h_k$, $\alpha$, $\tau$ and $q$ are defined
by \eqref{eq:lmk}, \eqref{def:hk} and \eqref{def:alp-tau}.
% 
% BK 22april09 vvvvv
% Moreover, we introduce the following assumption:
% \medskip
% 
% \noindent (A4) \
% The operators $F_i$ in \eqref{eq:inv-probl} and it's derivatives $F'_i$
% are Lipschitz continuous, i.e., there exists a constant $L$ such that
% $$
% \norm{F_i(x) - F_i(\bar x)} + \norm{F'_i(x) - F'_i(\bar x)} \ \le \
% L \, \norm{x - \bar x} \, , \quad
% {\rm for \ all} \ x , \bar x \in B_\rho(x_0).
% $$
% Moreover, the constants $\alpha$ in \eqref{def:alp-tau} and $C$ in
% \eqref{eq:a-dfb} satisfy $C^2 < \alpha$.
% BK 22april09 ^^^^^
% \medskip

Our main goal in this section is to prove that $x_{k_*^\delta}^\delta$
converges to a solution of \eqref{eq:inv-probl} as $\delta \to 0$,
where $k_*^\delta$ is defined in \eqref{def:discrep-lmk}.
The first step is to verify that, for noisy data, the stopping index
$k_*^\delta$ is well defined.

\begin{propo} \label{prop:st-fin-lmk}
Assume $\delta_{\rm min} := \min \{ \delta_0, \dots  \delta_{N-1} \} > 0$.
Then $k_*^\delta$ in \eqref{def:discrep-lmk} is finite, and the estimate
$k_*^\delta = O(\delta_{\rm min}^{-2})$ holds true. Moreover,
\begin{equation} \label{eq:lmk-res-discr}
\norm{ F_{i}(x_{k_*^\delta}^\delta) - y_i^\delta } \ < \
\tau \delta_i \, , \qquad i = 0, \dots, N-1 \, .
\end{equation}
\end{propo}
\begin{proof}
Assume by contradiction that for every $l \in \N$, there exists a
$i(l) \in \set{0,\dots, N-1}$ such that
$\norm{ F_{i(l)} (x_{lN + i(l)}) - y_{i(l)}^\delta } \ge \tau\delta_{i(l)}$.
From Proposition~\ref{prop:monot} it follows that \eqref{eq:monot-aux} holds
for $k=1, \dots, lN$. Summing over $k$, leads to
$$
- \norm{x_0 - x^*}^2
  \leq  \summ_{k=1}^{lN-1} \displaystyle 2 \frac{\omega_k}{\alpha}
        \norm{ B_k^\delta } \big[ (\eta - q) \,
        \norm{ F_{[k]}(x_{k+1}^\delta) - y_{[k]}^\delta } +
        (1 + \eta) \delta_{[k]} \big] \, , \quad l \in \N \, .
$$
Using the fact that either $\omega_k=0$ or
$\norm{ F_{[k]}(x_k^\delta) - y_{[k]}^\delta } > \tau \delta_{[k]}$,
we obtain
\begin{equation} \label{eq:st-finite-aux1}
\norm{x_0 - x^*}^2 \ \geq \
  \summ_{k=1}^{lN-1} \displaystyle 2 \frac{\omega_{k}}{\alpha}
  \norm{ B_k^\delta } \, \delta_{[k]} \, \big[ \tau (q-\eta) - (1+\eta) \big] \, .
\end{equation}
From equations \eqref{eq:monot-auxB}, \eqref{eq:st-finite-aux1} and the fact that
$\omega_{l'N + i(l')} = 1$
% $x_{l'N + i(l')} \not= x_{l'N}$
for all $l' \in \N$, we obtain
\begin{equation} \label{eq:st-finite-aux2}
\norm{x_0 - x^*}^2 \ \geq \
  \Big[ \tau (q-\eta) - (1+\eta) \Big]
  \, 2 l \, \frac{q \tau \delta_{\rm min}^2}{\alpha}
  \, , \quad l \in \N \, .
\end{equation}
Due to \eqref{def:alp-tau}, the right hand side of \eqref{eq:st-finite-aux2}
tends to $+\infty$ as $l \to \infty$, which gives a contradiction.
Consequently, the minimum in (\ref{def:discrep-lmk}) takes a finite value.

To prove $k_*^\delta = O(\delta_{\rm min}^{-2})$, it is enough to take
$l = k_*^\delta / N \in \mathbb N$ in \eqref{eq:st-finite-aux2} and estimate
$$
k_*^\delta \ \le \ \alpha N \, \norm{x_0 - x^*}^2 \,
                   \Big( 2 q \tau \, [\tau (q-\eta) - (1+\eta)]
                         \, \delta_{\rm min}^2 \Big)^{-1} \, .
$$

It remains to prove \eqref{eq:lmk-res-discr}. Assume to the contrary that
$\norm{ F_{i}(x_{k_*^\delta}^\delta) - y_i^\delta } \geq \tau \delta_i$
for some $i \in \set{0,\dots, N-1}$. From \eqref{def:omk} and
\eqref{def:discrep-lmk} it follows that $\omega_{k_*^\delta+i} = 1$ and
$x_{k_*^\delta+i+1}^\delta = x_{k_*^\delta+i}^\delta$ respectively.
Thus, from \eqref{eq:monot-aux} and \eqref{eq:monot-auxB} it follows that
\begin{eqnarray*}
0 & \le & 2 \alpha^{-1} \norm{ B_{k_*^\delta+i}^\delta } \,
   \big[ (\eta-q) \norm{ F_{i}(x_{k_*^\delta+i}^\delta) - y_{i}^\delta }
         + (1+\eta) \delta_{i} \big] \\
  & \le & 2 \alpha^{-1} \,
  \norm{ F_{i}(x_{k_*^\delta}^\delta) - y_{i}^\delta }^2
  \big[ (\eta-q) + (1+\eta) \tau^{-1} \big] \, .
\end{eqnarray*}
However, since $\norm{ F_{i}(x_{k_*^\delta}^\delta) - y_{i}^\delta } \geq
\tau \delta_{min} > 0$, the inequality above leads to
$[(\eta-q) + (1+\eta) \tau^{-1}] \geq 0$. This contradicts \eqref{def:alp-tau},
completing the proof of \eqref{eq:lmk-res-discr}.
\end{proof}

\begin{lemma} \label{lem:stabil-lmk}
Let $\delta_j = (\delta_{j,0}, \dots, \delta_{j,N-1}) \in (0,\infty)^N$
be given with $\lim_{j\to\infty} \delta_j = 0$. Moreover, let
$y^{\delta_j} = (y_0^{\delta_j}, \dots, y_{N-1}^{\delta_j}) \in Y^N$ be a
corresponding sequence of noisy data satisfying
$\norm{ y_i^{\delta_j} - y_i } \le \delta_{j,i}$, $i = 0, \dots, N-1$, $j \in \N$.
Then, for each fixed $k \in \N$ we have $\lim_{j\to\infty} x_k^{\delta_j} = x_k$.
\end{lemma}
\begin{proof}
The proof of Lemma~\ref{lem:stabil-lmk} uses an inductive argument
in $k$. First assume $k=0$ and notice that $x^{\delta_j}_0 = x_0$ for
$j\in\N$. 
Now, take $k > 0$ and assume that for all $k' \le k$ we have $\lim_{j\to\infty}
x_{k'}^{\delta_j} = x_{k'}$. Two cases must be considered:
If $\omega_k = 1$ we estimate
\begin{align} \label{eq:case-w1-ind-lmk}
&\norm{ x_{k+1}^{\delta_j} - x_{k+1} }^2 
\nonumber\\
& \le \norm{ x_k^{\delta_j} - x_k }
+ \norm{ h_k^{\delta_j} - h_k }
\nonumber\\
& \le \norm{ x_k^{\delta_j} - x_k }
\nonumber\\
& \quad + \norm{( F_{[k]}'(x_k^\delta)^* F_{[k]}'(x_k^\delta) + \alpha I)^{-1}
          F_{[k]}'(x_k^\delta)^* 
	-( F_{[k]}'(x_k)^* F_{[k]}'(x_k) + \alpha I)^{-1}
          F_{[k]}'(x_k)^*}
\nonumber\\
& \qquad  \cdot\norm{y_{[k]}^\delta - F_{[k]}(x_k^\delta)}
\nonumber\\
& \quad + \norm{( F_{[k]}'(x_k)^* F_{[k]}'(x_k) + \alpha I)^{-1}
          F_{[k]}'(x_k)^*} \cdot
	\norm{ y_{[k]}^\delta - y_{[k]} - F_{[k]}(x_k^\delta) + F_{[k]}(x_k)} 
\nonumber\\
& \le \norm{ x_k^{\delta_j} - x_k }
\nonumber\\
& \quad + 9/4 \,\alpha^{-1}\norm{ F_{[k]}'(x_k^\delta) - F_{[k]}'(x_k) }
	\norm{y_{[k]}^\delta - F_{[k]}(x_k^\delta)}
\nonumber\\
& \quad + 1/2 \,\alpha^{-1/2} \
	\big\{ \delta_{[k]} 
	+ \norm{ F_{[k]}(x_k^\delta) - F_{[k]}(x_k) } \big\}\,,
\end{align}
where we have used the identity
\begin{align*}
&(A^*A+\alpha I)^{-1} A^* - (B^*B+\alpha I)^{-1} B^*\\
& = (A^*A+\alpha I)^{-1} (A^*-B^*)
   + (A^*A+\alpha I)^{-1}(B^*B-A^*A)(B^*B+\alpha I)^{-1} B^*\\
& = (A^*A+\alpha I)^{-1} (A-B)^* 
   + (A^*A+\alpha I)^{-1}(A^*(B-A)+(B^*-A^*)B)(B^*B+\alpha I)^{-1} B^*
\end{align*}
and the estimates
$$ 
\norm{(A^*A+\alpha I)^{-1}}\leq \alpha^{-1}\,, \quad 
\norm{(A^*A+\alpha I)^{-1}A^*}\leq 1/2 \,\alpha^{-1/2}\,, \quad 
\norm{A(A^*A+\alpha I)^{-1}A^*}\leq 1\,.
$$
for linear operators $A$, $B$.
Otherwise, if $\omega_k = 0$ we have $x_{k+1}^{\delta_j} = x_k^{\delta_j}$ and
$\norm{ F_{[k]}(x_k^{\delta_j}) - y_{[k]}^{\delta_j} } \le \tau \delta_{k,j}$.
Therefore,
\begin{align} \label{eq:case-w0-ind-lmk}
\norm{ x_{k+1}^{\delta_j} & - x_{k+1} }
\  =  \ \norm{ x_k^{\delta_j} - (x_k + h_k) }
\ \le \ \norm{ x_k^{\delta_j} - x_k } 
	+ 1/2 \,\alpha^{-1/2} \norm{ F_{[k]}(x_k) - y_{[k]} }
        \nonumber \\
& \le \ \norm{ x_k^{\delta_j} - x_k } 
	+ 1/2 \,\alpha^{-1/2}\big\{
        \norm{ F_{[k]}(x_k) - F_{[k]}(x_k^{\delta_j}) } +
        \norm{ F_{[k]}(x_k^{\delta_j}) - y_{[k]}^{\delta_j} } +
        \norm{ y_{[k]}^{\delta_j} - y_{[k]} } \big\} \nonumber \\
& \le \ \norm{ x_k^{\delta_j} - x_k } 
	+ 1/2 \,\alpha^{-1/2} \big\{
        \norm{ F_{[k]}(x_k) - F_{[k]}(x_k^{\delta_j}) }
	 + (\tau+1) \delta_{k,j} \big\} \, .
\end{align}
Thus, it follows from \eqref{eq:case-w1-ind-lmk}, \eqref{eq:case-w0-ind-lmk},
the continuity of $F_{[k]}$, $F_{[k]}'$, and the induction hypothesis that
$\lim\limits_{j\to\infty} x_{k+1}^{\delta_j} = x_{k+1}$.
\end{proof}

Now we are ready to state a semi-convergence result for the loping
Levenberg-Marquardt-Kaczmarz iteration. For the proof, both
Proposition~\ref{prop:st-fin-lmk} and Lemma~\ref{lem:stabil-lmk} are required.

\begin{theorem}[Convergence for noisy data] \label{th:semiconv-lmk}
Let $\delta_j = (\delta_{j,0}, \dots$, $\delta_{j,N-1})$ be a given
sequence in $(0,\infty)^N$ with $\lim_{j\to\infty} \delta_j = 0$, and
let $y^{\delta_j} = (y^{\delta_j}_{0}, \dots, y^{\delta_j}_{N-1}) \in Y^N$ be a
corresponding sequence of noisy data satisfying
$$
\norm{ y_i^{\delta_j} - y_i } \le \delta_{j,i} \, , \quad
i = 0, \dots, N-1 \, , \ j \in \N \, .
$$
Denote by $k^j_* := k_*(\delta_j, y^{\delta_j})$ the corresponding stopping
index defined in \eqref{def:discrep-lmk}. Then $x_{k^j_*}^{\delta_j}$
converges to a solution 
% $x^*$ %%%%(Siehe oben)
$x^+$ of \eqref{eq:inv-probl}.
Moreover, if \eqref{eq:kern-cond} holds, then $x_{k^j_*}^{\delta_j}$
converges to $x^\dag$.
\end{theorem}
\begin{proof}
The proof follows the lines of \cite[Theor.3.6]{CHLS08} (see also
\cite{HanNeuSch95}) and is divided into two cases. For the first case
Proposition~\ref{prop:st-fin-lmk} and Lemma~\ref{lem:stabil-lmk} are needed.
For the second case, we need Proposition~\ref{prop:monot},
Theorem~\ref{th:exact} and Lemma~\ref{lem:stabil-lmk}.
\end{proof}

% --------------------------------------------------------------------
\section{Numerical experiment} \label{sec:numeric}

% -------------------------------------
\subsection{Description of the model problem}

In this section we introduce a model which plays a key rule in
inverse doping problems related to measurements of the current flow,
namely the {\em linearized stationary bipolar case close to equilibrium}.

This model is obtained from the drift diffusion equations by linearizing
the Voltage-Current (VC) map at $U \equiv 0$ \cite{Le06,BELM04},
where the function $U = U(x)$ denotes the applied potential to the
semiconductor device.
This simplification is motivated by the fact that, due to hysteresis effects
for large applied voltage, the VC-map can only be defined as a single-valued
function in a neighborhood of $U=0$.
Additionally, we assume that the electron mobility $\mu_n(x) = \mu_n > 0$
and hole mobility $\mu_p(x) = \mu_p > 0$ are constant and that no
recombination-generation rate is present \cite{LMZ06,LMZ06a}.

Under these assumptions the Gateaux derivative of the VC-map $\Sigma_C$ at
the point $U=0$ in the direction $h \in H^{3/2}(\partial\Omega_D)$ is given
by the expression
\begin{equation} \label{eq:def-sigma-prime-C}
\Sigma'_C(0) h = \int_{\Gamma_1}
\left( \mu_n \, e^{V_{\rm bi}} \hat{u}_\nu -
\mu_p \, e^{-V_{\rm bi}} \hat{v}_\nu \right) \, ds ,
\end{equation}
where the concentrations of electrons and holes $(\hat{u}, \hat{v})$
(written in terms of the Slotboom variables) solve
\begin{subequations}  \label{eq:bipol-stat}  \begin{eqnarray}
{\rm div}\, (\mu_n e^{V^0} \nabla \hat{u})   & \hskip-1.7cm
 = \ 0                \label{eq:bipol-statA} & {\rm in}\ \Omega \\
{\rm div}\, (\mu_p e^{-V^0} \nabla \hat{v})  & \hskip-1.7cm
 = \ 0                \label{eq:bipol-statB} & {\rm in}\ \Omega \\
\hat{u} & \hskip-0.1cm
 = \ -\hat{v} \ = \ -h                       & {\rm on}\ \partial\Omega_D \\
\nabla\hat{u} \cdot\nu &
 = \ \nabla\hat{v} \cdot\nu \ = \ 0          & {\rm on}\ \partial\Omega_N
\end{eqnarray} \end{subequations}
and the potential $V^0$ is the solution of the thermal equilibrium problem
\begin{subequations}  \label{eq:equil-case} \begin{eqnarray}
\lambda^2 \, \Delta V^0 &                       \label{eq:equil-caseA}
   = \ e^{V^0} - e^{-V^0} - C(x)              & {\rm in}\ \Omega \\
V^0 & \hskip-2.1cm                              \label{eq:equil-caseB}
   = \ V_{\rm bi}(x)                          & {\rm on}\ \partial\Omega_D \\
\nabla V^0 \cdot \nu & \hskip-2.95cm            \label{eq:equil-caseC}
   = \ 0                                      & {\rm on}\ \partial\Omega_N \, .
\end{eqnarray} \end{subequations}
Here $\Omega \subset \R^d$ is a domain representing the semiconductor
device; the boundary $\partial\Omega$ of $\Omega$ is divided into two
nonempty disjoint parts: $\partial\Omega = \overline{\partial\Omega_N} \cup
\overline{\partial\Omega_D}$. The Dirichlet part of the boundary
$\partial\Omega_D$ models the Ohmic contacts, where the potential $V$ as
well as the concentrations $\hat{u}$ and $\hat{v}$ are prescribed;
the Neumann part $\partial\Omega_N$
of the boundary corresponds to insulating surfaces, thus a zero current
flow and a zero electric field in the normal direction are prescribed;
the Dirichlet part of the boundary splits into $\partial\Omega_D =
\Gamma_0 \cup \Gamma_1$, where the disjoint boundary parts $\Gamma_i$,
$i=0,1$, correspond to distinct contacts (differences in $U(x)$ between
different segments of $\partial\Omega_D$ correspond to the applied bias
between these two contacts).

The function $C(x)$ is the {\em doping profile} and models a preconcentration
of ions in the crystal, so $C(x) = C_{+}(x) - C_{-}(x)$ holds, where $C_{+}$
and $C_{-}$ are concentrations of negative and positive ions respectively.
In those subregions of $\Omega$ in which the preconcentration of negative
ions predominate (P-regions), we have $C(x) < 0$. Analogously, we define the
N-regions, where $C(x) > 0$ holds.
The boundaries between the P-regions and N-regions (where $C$ changes sign)
are called {\em pn-junctions}. Moreover, $V_{\rm bi}$ is a given logarithmic
function \cite{BELM04}.

% -------------------------------------
\subsection{Inverse doping problem}\label{sec:numerik}

The inverse problem we are concerned with consists in determining the
doping profile function $C$ in \eqref{eq:equil-case} from measurements
of the linearized VC-map $\Sigma'_C(0)$ in \eqref{eq:def-sigma-prime-C}.
Notice that we can split the inverse problem in two parts: The first
step is to define the function $\gamma(x) := e^{V^0(x)}$, $x \in \Omega$,
and solve the parameter identification problem
\begin{equation} \label{eq:num-d2nB}
\begin{array}{r@{\ }c@{\ }l@{\ }l}
   {\rm div}\, (\mu_n \gamma \nabla \hat u) & = & 0 &  {\rm in}\ \Omega \\
   \hat u & = & - U(x) & {\rm on}\ \partial\Omega_D \\
   \nabla \hat u \cdot \nu & = & 0 & {\rm on}\ \partial\Omega_N
\end{array}
\hskip0.6cm
\begin{array}{r@{\ }c@{\ }l@{\ }l}
   {\rm div}\, (\mu_p \gamma^{-1} \nabla \hat v) & = & 0 &  {\rm in}\ \Omega \\
   \hat v & = & U(x) & {\rm on}\ \partial\Omega_D \\
   \nabla \hat v \cdot \nu & = & 0 & {\rm on}\ \partial\Omega_N
\end{array}
\end{equation}
for $\gamma$, from measurements of $[\Sigma'_C(0)](U) = \int_{\Gamma_1}
( \mu_n \gamma \hat{u}_\nu - \mu_p \gamma^{-1} \hat{v}_\nu ) \, ds$.
The second step consists in the determination of the doping profile in
$ C(x) = \gamma(x) - \gamma^{-1}(x) - \lambda^2 \Delta (\ln \gamma(x))$,
$x \in \Omega$.
Since the evaluation of $C$ from $\gamma$ can be explicitly performed
in a stable way, we shall focus on the problem of identifying the function
parameter $\gamma$ in (\ref{eq:num-d2nB}).

Summarizing, the inverse doping profile problem in the linearized stationary
bipolar model (close to equilibrium) for pointwise measurements of the current
density reduces to the identification of the parameter $\gamma$ in
(\ref{eq:num-d2nB}) from measurements of the DN map
$$  \Lambda_\gamma : \begin{array}[t]{rcl}
    H^{3/2}(\partial\Omega_D) & \to & \R \, . \\
    U & \mapsto & \int_{\Gamma_1}( \mu_n \gamma \hat{u}_\nu
                  - \mu_p \gamma^{-1} \hat{v}_\nu ) \, ds
    \end{array} $$
In the formulation of the inverse problem we shall take into account some
restrictions imposed by the practical experiments, namely
\begin{itemize}
\item[{\em i)}] The voltage profile $U \in H^{3/2}(\partial\Omega_D)$ must
satisfy $U |_{\Gamma_1} = 0$ (in practice, $U$ is chosen to be piecewise
constant on the contact $\Gamma_1$ and to vanish on $\Gamma_0$);
\item[{\em ii)}] The identification of $\gamma$ has to be performed from a
finite number $N \in \N$ of measurements, i.e. from the data
$\big\{ (U_i, \Lambda_\gamma(U_i)) \big\}_{i=0}^{N-1} \in
\big[ H^{3/2}(\Gamma_0) \times \R \big]^N$.
\end{itemize}
Therefore, we can write this particular inverse doping profile problem
in the abstract formulation of system \eqref{eq:inv-probl}, namely
\begin{equation} \label{eq:ip-abstract}
 F_i(\gamma) \ = \ \Lambda_\gamma(U_i) \ =: \ y_i \, , \ i=0, \dots, N-1 \, ,
\end{equation}
where $U_i$ are fixed voltage profiles chosen as above; 
$X := L^2(\Omega) \supset D(F_i) := \{ \gamma \in L^\infty(\Omega)$;
$0 < \gamma_m \le \gamma(x) \le \gamma_M$, a.e. in $\Omega \}$;
$Y := \R$.

To the best of our knowledge, assumptions (A1) - (A3) are not satisfied for
the Dirichlet-to-Neumann operator $\Lambda_\gamma$. Therefore, although the
operators $F_i$ are continuous \cite{BELM04}, the analytical convergence
results of the previous sections do not apply for system \eqref{eq:ip-abstract}.

In the following numerical experiment we assume that nine measurements are
available, i.e. $N = 9$, in \eqref{eq:ip-abstract}.
The domain $\Omega \subset \mathbb R^2$ is the unit square, and the boundary
parts are defined as follows
$$ \Gamma_1 \ := \  \{ (x,1) \, ;\ x \in (0,1) \} \, , \ \
   \Gamma_0 \ := \  \{ (x,0) \, ;\ x \in (0,1) \} \, , $$
$$ \partial\Omega_N \ := \ \{ (0,y) \, ;\ y \in (0,1) \} \cup 
   \{ (1,y) \, ;\ y \in (0,1) \} \, . $$
The fixed inputs $U_i$, are chosen to be piecewise constant functions
supported in $\Gamma_0$
$$  U_i(x) \ := \ \left\{ \begin{array}{rl}
      1, & |x - x_i| \le 2^{-4} \\
      0, & {\rm else} \end{array} \right. \,
    i=0,\dots, N-1 \, , $$
where the points $x_i$ are uniformly spaced in $[0,1]$.
The parameter $\gamma$ to be identified is shown in Figure~%
\ref{fig:exsol-source}~(a) (notice that $\gamma(x) \in [ 0, 10 ]$ a.e. in $\Omega$). In Figure~\ref{fig:exsol-source}~(b) a typical
voltage source $U_i$ (applied at $\Gamma_0$) and the corresponding
solution $\hat u$ of (\ref{eq:num-d2nB}) are shown.
In these two pictures, as well as in the forthcoming ones, $\Gamma_1$ is the
lower left edge and $\Gamma_0$ is the top right edge (the origin corresponds
to the upper right corner).

\begin{figure}[ht!]
\centerline{
\includegraphics[width=7.5cm]{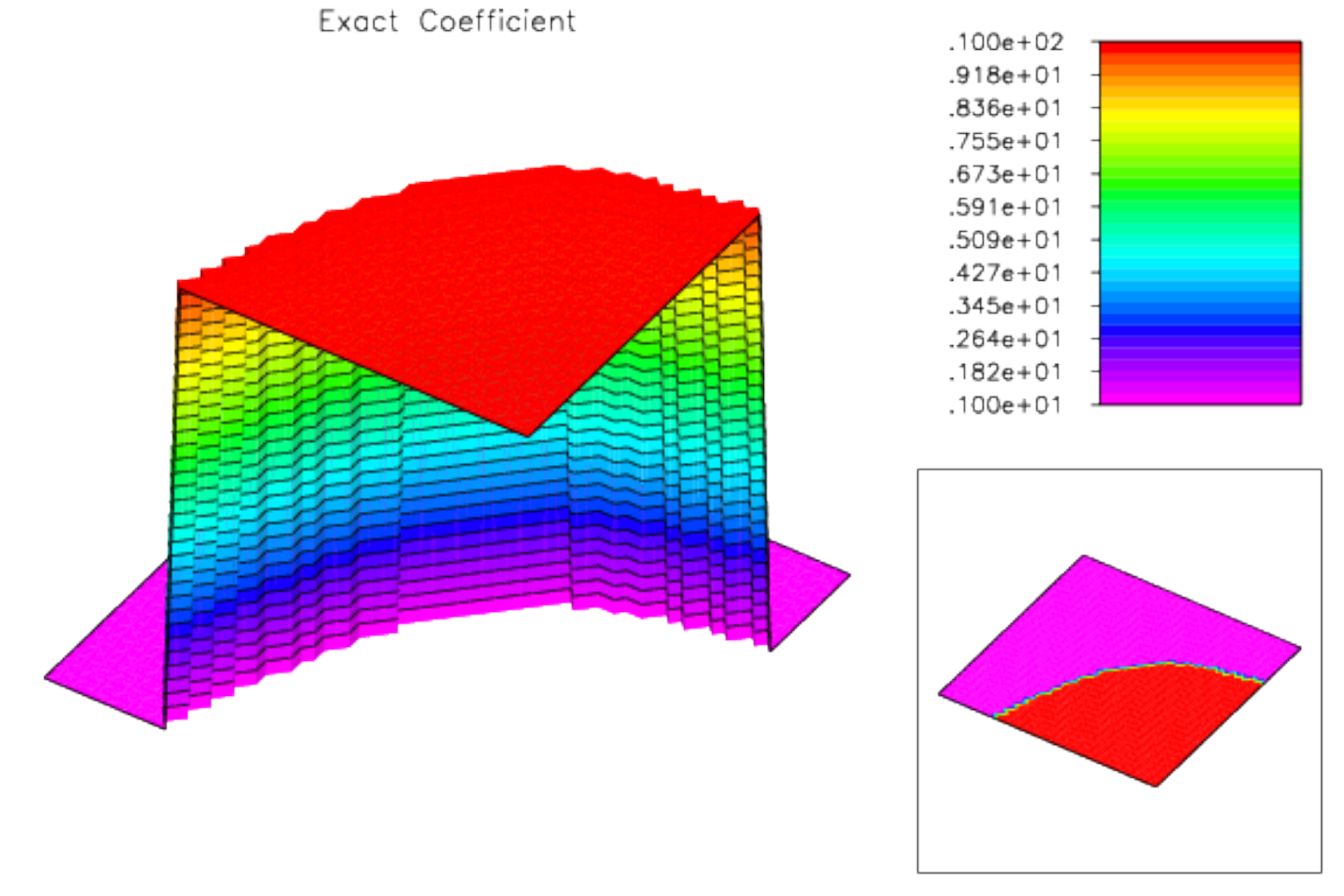} \hfil
\includegraphics[width=7.5cm]{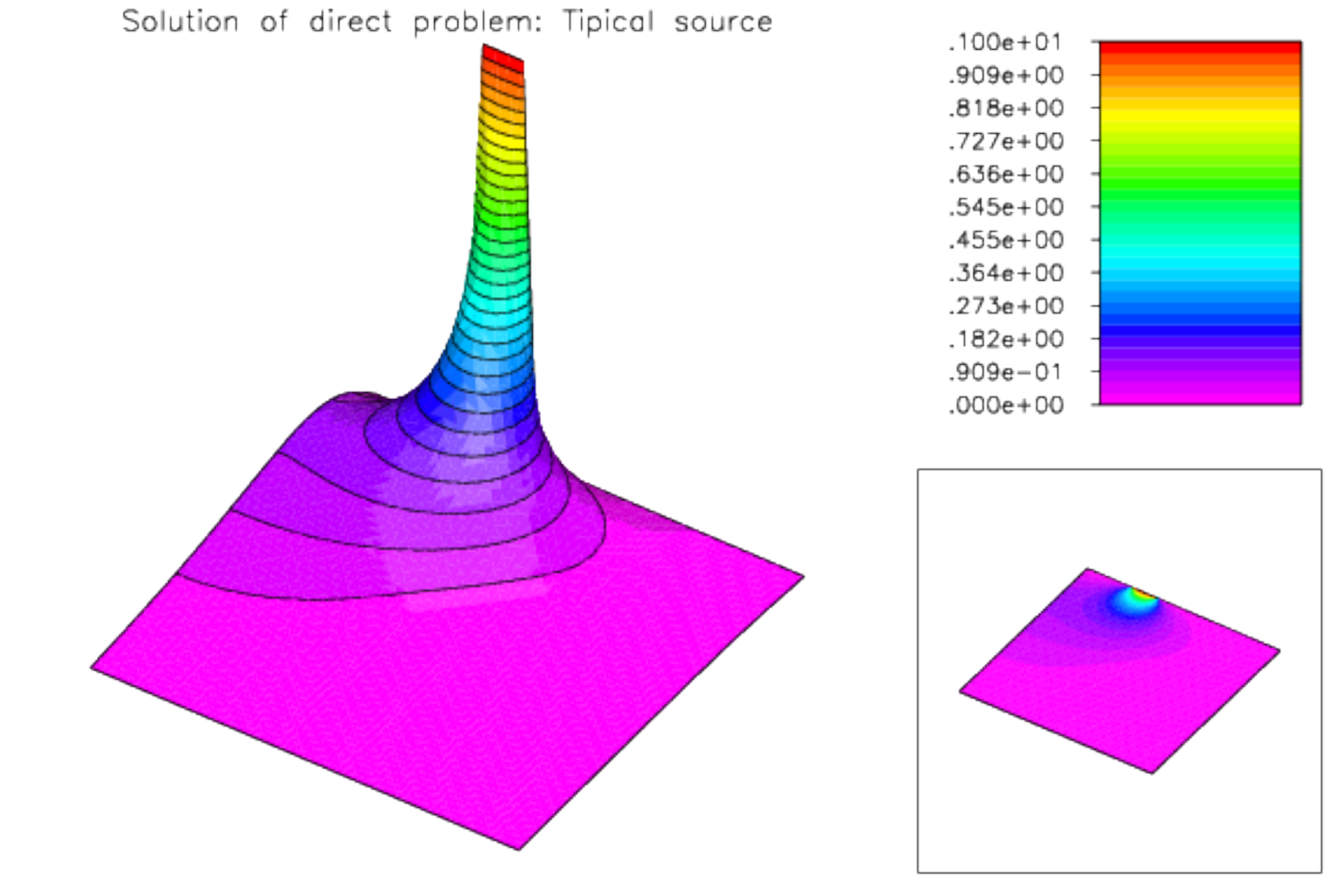} }
\centerline{\hfil (a) \hskip7cm (b) \hfil}
\caption{\small In picture (a) the parameter $\gamma$ to be identified is
shown. In picture (b) a typical voltage source $U_i$ (boundary condition)
and the corresponding solution $\hat u$ of (\ref{eq:num-d2nB}) are shown.}
\label{fig:exsol-source}
\end{figure}

For comparison purposes we implemented both the \textsc{l-LMK} and the 
\textsc{l-LK} iteration. The initial condition for both methods is presented
in Figure~\ref{fig:lLK-lLMK}~(c).
% BK 05jul09 vvvvv
The linear system in the \textsc{l-LMK} is solved inexactly by three CG steps, so the numerical effort for one step of the \textsc{l-LMK} is three times the one for one step of the \textsc{l-LK}.
In the computations it turned out that the performance of the \textsc{l-LMK} is not very sensitive to the value of $\alpha$.
% BK 05jul09 ^^^^^
The ``exact`` data $y_i$, $i=0,\dots,8$, were obtained by solving the direct
problems \eqref{eq:num-d2nB} using a finite element type method and adaptive
mesh refinement (approx 8000 elements). Artificially generated (random) noise
of 5\% was introduced to $y_i$ in order to generate the noisy data $y_i^\delta$
for the inverse problem.
In order to avoid inverse crimes, a coarser grid (with approx 2000 elements)
was used in the finite element method to implement the \textsc{l-LMK} and
\textsc{l-LK} iterations.

For both iterative methods the same stopping rule \eqref{def:discrep-lmk}
was used. We assumed exact knowledge of the noise level and chose $\tau = 2$.
In Figure~\ref{fig:lLK-lLMK}~(d) we plot, for each one of the iterations, the
number of non-loped inner steps in each cycle. For the \textsc{l-LMK} iteration
(solid red line) the stopping criterion is achieved after 24 cycles, while the
\textsc{l-LK} iteration (dashed blue line) is stopped after 205 cycles. In this
picture one also observes that the computational effort to perform the
\textsc{l-LMK} cycles decreases much faster than in the \textsc{l-LK} iteration.

The quality of the final result obtained with the \textsc{l-LMK} method can
be seen at Figure~\ref{fig:lLK-lLMK}~(a), where the iteration error for the
approximation obtained after 24 cycles is depicted.
In Figure~\ref{fig:lLK-lLMK}~(b) we present the iteration error for the
\textsc{l-LK} iteration after 205 cycles, when the stopping criterion is reached.

Since the same noisy data and the same stopping rule were used for both
iterations, the quality of the final results in Figures~\ref{fig:lLK-lLMK}~(a)
and~(b) is similar. However, the \textsc{l-LMK} iteration needed a
much smaller number of cycles to reach the stopping criterion than the \textsc{l-LK}
iteration (starting from the same initial guess). Moreover, the number of
actually performed inner steps per cycle is much smaller for the \textsc{l-LMK}
iteration.
All these observations lead us to conclude that the \textsc{l-LMK} iteration
is numerically much more efficient than the \textsc{l-LK} iteration.

\begin{figure}[ht!]
\centerline{\includegraphics[width=7.5cm]{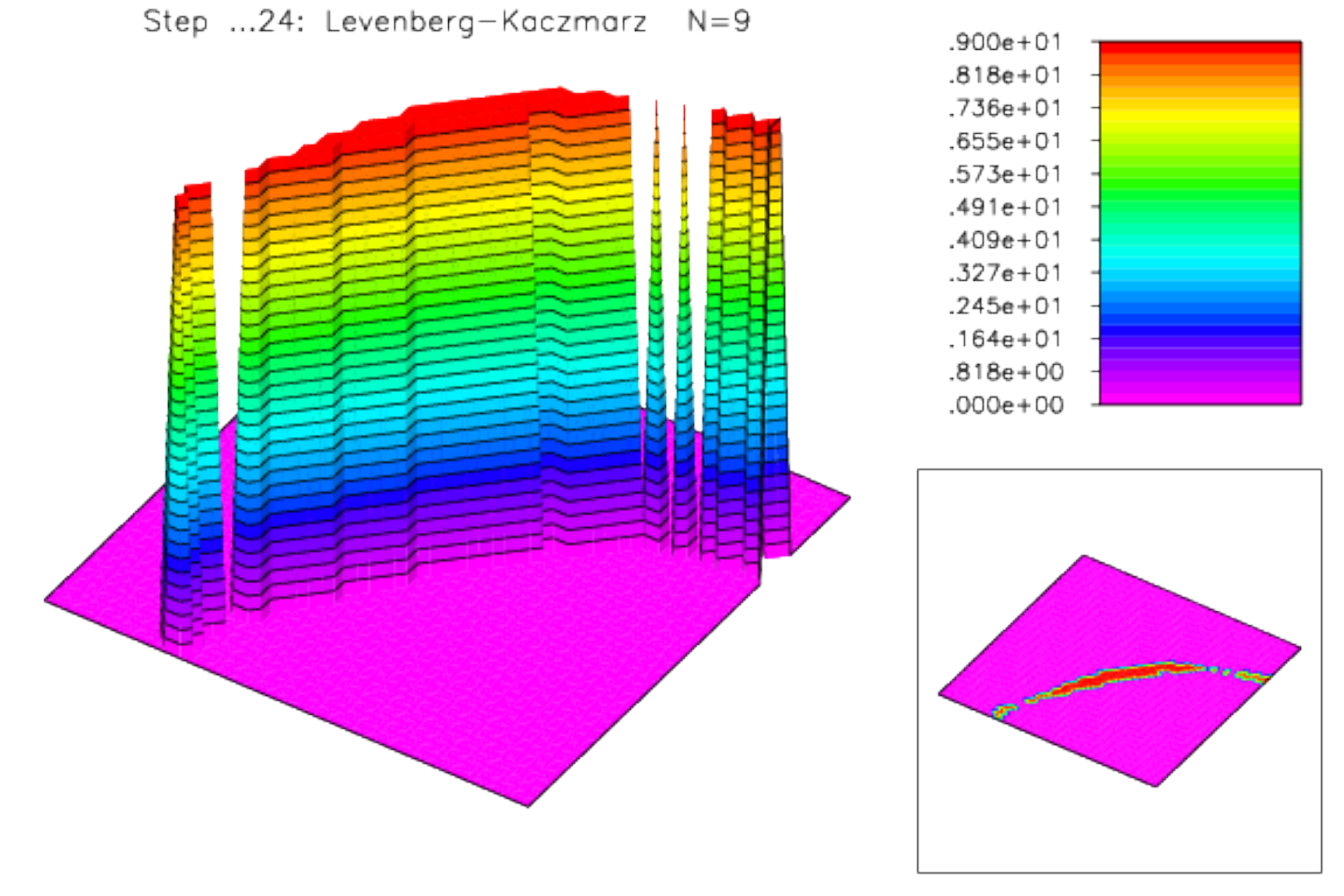} \hfil
            \includegraphics[width=7.5cm]{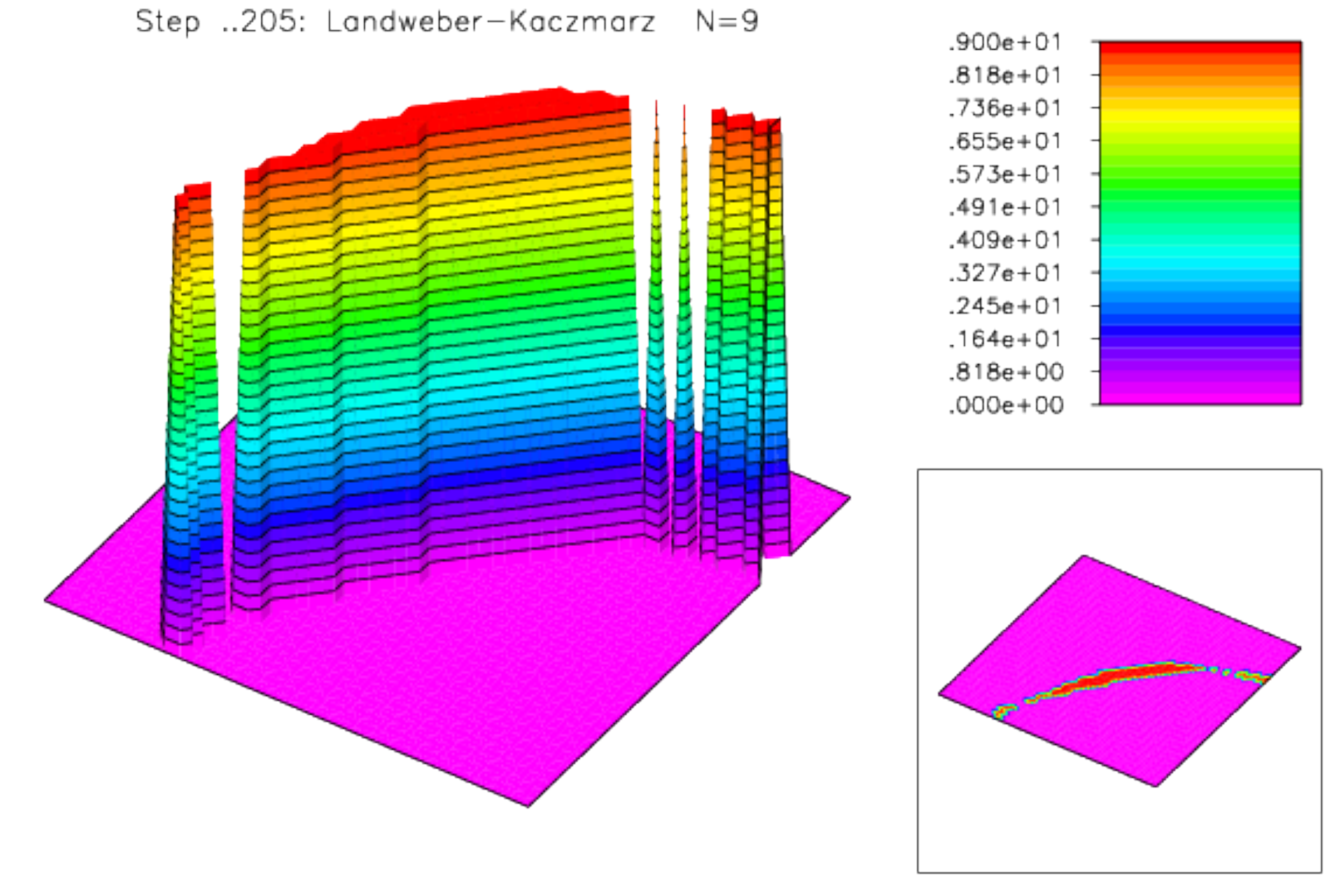} }
\centerline{\hfil (a) \hskip7cm  (b) \hfil}
\medskip
\centerline{\includegraphics[width=7.5cm]{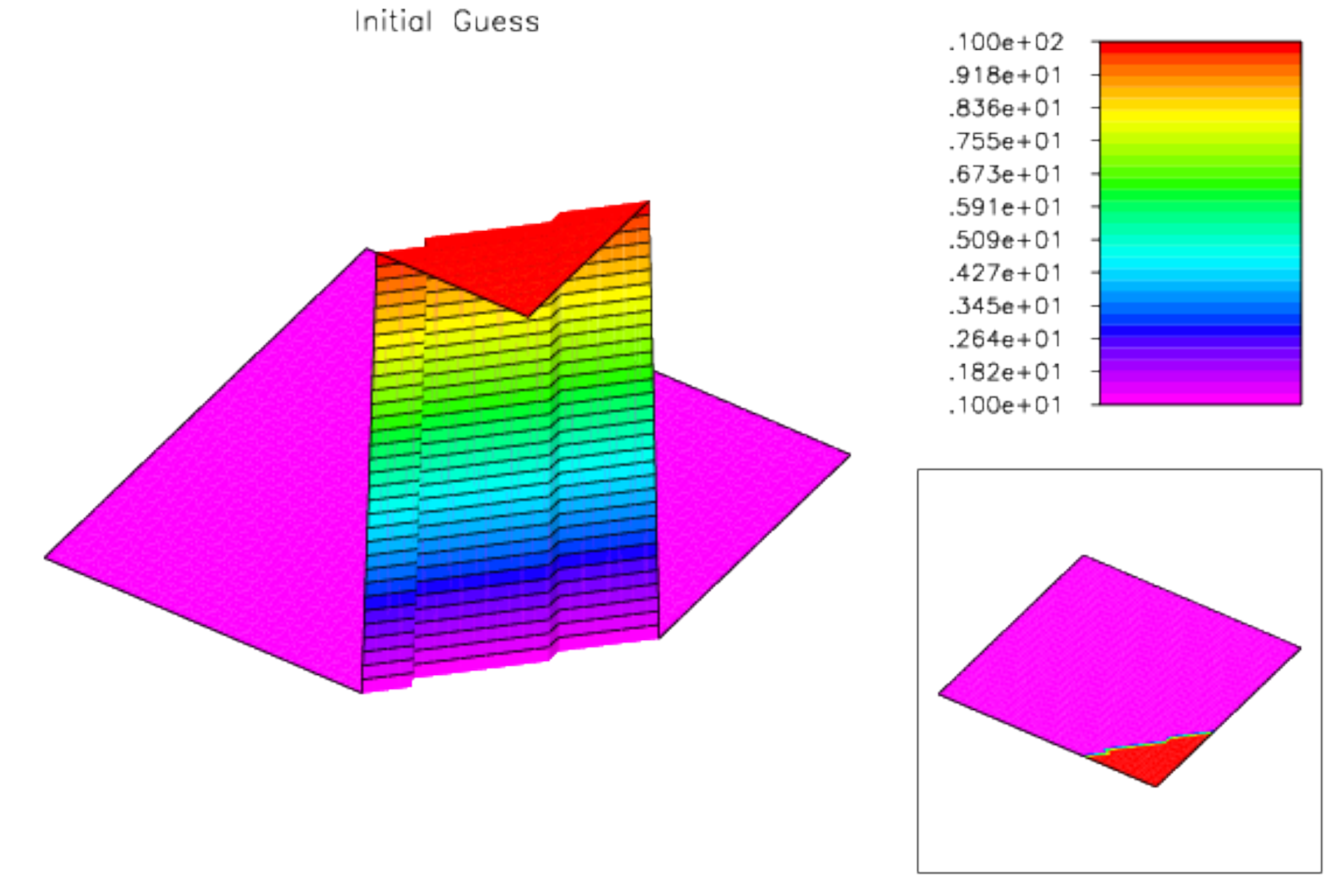} \hfil
            \includegraphics[width=7.5cm,height=5cm]{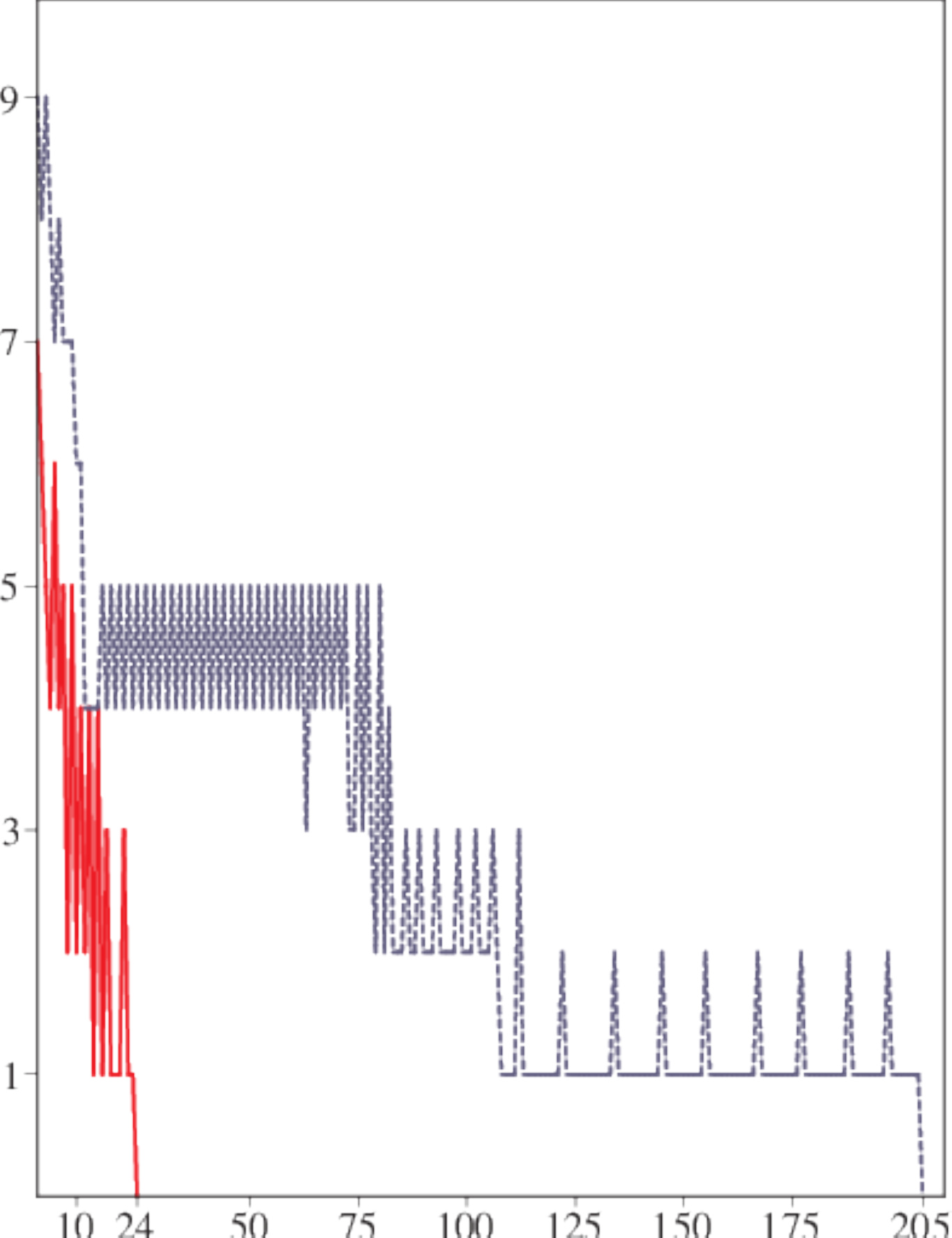} }
\centerline{\hfil (c) \hskip7cm  (d) \hfil}
\caption{\small Numerical experiment with noisy data:
On picture (a) the iterative error obtained with the \textsc{l-LMK} method
after 24 cycles. On picture (b) the iterative error obtained with the
\textsc{l-LK} method after 205 cycles.
On picture (c) the initial condition for both iterative methods.
On picture (d) the number of non-loped inner steps in each cycle for
\textsc{l-LMK} (solid red) and \textsc{l-LK} (dashed blue), respectively.} 
\label{fig:lLK-lLMK}
\end{figure}

% --------------------------------------------------------------------
\section{Conclusions} \label{sec:conclusion}

In this article we propose a new iterative method for inverse problems
of the form \eqref{eq:inv-probl}, namely the \textsc{l-LMK} iteration.
In the case of exact data this method reduces to the \textsc{LMK} iteration.

In the \textsc{l-LMK} iteration we omit an update of the \textsc{LMK}
iteration (within one cycle) if the corresponding $i$-th residual is
below some threshold. Consequently, the \textsc{l-LMK} method is not
stopped until all residuals are below the specified threshold. We provide
a complete convergence analysis for the \textsc{l-LMK} iteration, proving
that it is a convergent regularization method in the sense of
\cite{EngHanNeu96}.
Moreover, we provide a numerical experiment for a nonlinear inverse doping
problem and observe that the \textsc{l-LMK} iteration generates results
that are comparable with other Kaczmarz type iterations. The
specific example considered in Section~\ref{sec:numerik} indicates that
the \textsc{l-LMK} iteration is numerically more efficient than the
\textsc{l-LK} iteration.

% --------------------------------------------------------------------
\section*{Acknowledgments}

B.K. acknowledges support from the Stuttgart Research Center for Simulation
Technology and Cluster of Excellence ``Simulation Technology'' (SimTech). \\
The work of A.L. is supported by the Brazilian National Research Council CNPq,
grants 306020/2006--8, 474593/2007--0; and by the Alexander von Humboldt
Foundation AvH.

% --------------------------------------------------------------------
\bibliographystyle{amsplain}

\bibliography{kaczLM}

\providecommand{\bysame}{\leavevmode\hbox to3em{\hrulefill}\thinspace}
\providecommand{\MR}{\relax\ifhmode\unskip\space\fi MR }
\providecommand{\MRhref}[2]{%
  \href{http://www.ams.org/mathscinet-getitem?mr=#1}{#2}
}
\providecommand{\href}[2]{#2}
\begin{thebibliography}{10}

\bibitem{BakKok04}
A.B. Bakushinsky and M.Y. Kokurin, \emph{Iterative methods for approximate
  solution of inverse problems}, Mathematics and Its Applications, vol. 577,
  Springer, Dordrecht, 2004.

\bibitem{BELM04}
M.~Burger, H.~W. Engl, A.~Leit\~ao, and P.A. Markowich, \emph{On inverse
  problems for semiconductor equations}, Milan J. Math. \textbf{72} (2004),
  273--313.

\bibitem{BK06}
M.~Burger and B.~Kaltenbacher, \emph{Regularizing {N}ewton-{K}aczmarz methods
  for nonlinear ill-posed problems}, SIAM J. Numer. Anal. \textbf{44} (2006),
  153--182.

\bibitem{CHLS08}
A.~De~Cezaro, M.~Haltmeier, A.~Leit{\~a}o, and O.~Scherzer, \emph{On
  steepest-descent-{K}aczmarz methods for regularizing systems of nonlinear
  ill-posed equations}, Appl. Math. Comput. \textbf{202} (2008), no.~2,
  596--607.

\bibitem{DES98}
P.~Deuflhard, H.W. Engl, and O.~Scherzer, \emph{A convergence analysis of
  iterative methods for the solution of nonlinear ill--posed problems under
  affinely invariant conditions}, Inverse Problems \textbf{14} (1998),
  1081--1106.

\bibitem{EngHanNeu96}
H.W. Engl, M.~Hanke, and A.~Neubauer, \emph{Regularization of inverse
  problems}, Kluwer Academic Publishers, Dordrecht, 1996.

\bibitem{Gr84}
C.~W. Groetsch, \emph{The theory of {T}ikhonov regularization for {F}redholm
  equations of the first kind}, Research Notes in Mathematics, vol. 105, Pitman
  (Advanced Publishing Program), Boston, MA, 1984.

\bibitem{HLS07}
M.~Haltmeier, A.~Leit\~ao, and O.~Scherzer, \emph{Kaczmarz methods for
  regularizing nonlinear ill-posed equations. {I}. convergence analysis},
  Inverse Probl. Imaging \textbf{1} (2007), no.~2, 289--298.

\bibitem{HLR09}
M.~Haltmeier, A.~Leit{\~a}o, and E.~Resmerita, \emph{On regularization methods
  of {E}{M}-{K}aczmarz type}, Inverse Problems \textbf{25} (2009), 075008.

\bibitem{Ha97}
M.~Hanke, \emph{A regularizing {L}evenberg-{M}arquardt scheme, with
  applications to inverse groundwater filtration problems}, Inverse Problems
  \textbf{13} (1997), no.~1, 79--95.

\bibitem{HanNeuSch95}
M.~Hanke, A.~Neubauer, and O.~Scherzer, \emph{A convergence analysis of
  {L}andweber iteration for nonlinear ill-posed problems}, Numer. Math.
  \textbf{72} (1995), 21--37.

\bibitem{KalNeuSch08}
B.~Kaltenbacher, A.~Neubauer, and O.~Scherzer, \emph{Iterative regularization
  methods for nonlinear ill-posed problems}, Radon Series on Computational and
  Applied Mathematics, vol.~6, Walter de Gruyter GmbH \& Co. KG, Berlin, 2008.

\bibitem{Lan51}
L.~Landweber, \emph{An iteration formula for {F}redholm integral equations of
  the first kind}, Amer. J. Math. \textbf{73} (1951), 615--624.

\bibitem{Le06}
A.~Leit{\~a}o, \emph{Semiconductors and {D}irichlet-to-{N}eumann maps}, Comput.
  Appl. Math. \textbf{25} (2006), no.~2-3, 187--203.

\bibitem{LMZ06a}
A.~Leitao, P.A. Markowich, and J.P. Zubelli, \emph{Inverse problems for
  semiconductors: Models and methods}, ch.~in Transport Phenomena and Kinetic
  Theory: Applications to Gases, Semiconductors, Photons, and Biological
  Systems, Ed. C.Cercignani and E.Gabetta, Birkh\"auser, Boston, 2006.

\bibitem{LMZ06}
\bysame, \emph{On inverse dopping profile problems for the stationary
  voltage-current map}, Inv.Probl. \textbf{22} (2006), 1071--1088.

\bibitem{Le44}
K.~Levenberg, \emph{A method for the solution of certain non-linear problems in
  least squares}, Quart. Appl. Math. \textbf{2} (1944), 164--168.

\bibitem{Ma63}
D.W. Marquardt, \emph{An algorithm for least-squares estimation of nonlinear
  parameters}, J. Soc. Indust. Appl. Math. \textbf{11} (1963), 431--441.

\bibitem{Mor93}
V.A. Morozov, \emph{Regularization methods for ill--posed problems}, CRC Press,
  Boca Raton, 1993.

\bibitem{Sch93a}
O.~Scherzer, \emph{Convergence rates of iterated {T}ikhonov regularized
  solutions of nonlinear ill-posed problems}, Numer. Math. \textbf{66} (1993),
  no.~2, 259--279.

\bibitem{Sch96}
\bysame, \emph{A convergence analysis of a method of steepest descent and a
  two-step algorithm for nonlinear ill-posed problems}, Numer. Funct. Anal.
  Optim. \textbf{17} (1996), no.~1-2, 197--214.

\bibitem{SeiVog89}
T.I. Seidman and C.R. Vogel, \emph{Well posedness and convergence of some
  regularisation methods for non--linear ill posed problems}, Inverse Probl.
  \textbf{5} (1989), 227--238.

\bibitem{Tik63b}
A.N. Tikhonov, \emph{Regularization of incorrectly posed problems}, Soviet
  Math. Dokl. \textbf{4} (1963), 1624--1627.

\bibitem{TikArs77}
A.N. Tikhonov and V.Y. Arsenin, \emph{Solutions of ill-posed problems}, John
  Wiley \& Sons, Washington, D.C., 1977, Translation editor: Fritz John.

\end{thebibliography}

\end{document}